\documentclass[reqno, 11pt]{amsart}
\usepackage{amsmath,mathtools}
 \usepackage{amssymb}
\usepackage{amsthm}
\usepackage{amsmath}
\usepackage{times}
\usepackage{latexsym}
\usepackage[mathscr]{eucal}
\usepackage{graphicx}
\usepackage{float}

\numberwithin{equation}{section}
 
  \newtheorem{theorem}{Theorem}[section]
  \newtheorem{proposition}[theorem]{Proposition}
  \newtheorem{lemma}[theorem]{Lemma}
  \newtheorem{corollary}[theorem]{Corollary}

  \newtheorem{remark}[theorem]{Remark}
  
  \newtheorem{definition}[theorem]{Definition}

\title[Non-existence of lightlike hypersurfaces]{Non-existence of certain lightlike hypersurfaces of an indefinite Sasakian manifold}
\author[Samuel Ssekajja]{Samuel Ssekajja}
\newcommand{\acr}{\newline\indent}
\address{ School of Mathematics\acr
 University of the Witwatersrand\acr
 Private Bag 3, Wits 2050\acr
South Africa}
\email{samuel.ssekajja@wits.ac.za} 
\thanks{}


\begin{document}

\begin{abstract}
Here, we consider a lightlike hypersurface, tangent to the structure vector field, of an indefinite Sasakian manifold. We prove that no such a hypersurface can either have parallel or recurrent second fundamental forms. In addition to the above, we also prove that no such a hypersurface may have parallel or recurrent induced structural tensors.
\end{abstract}
\maketitle
\section{Introduction} 

Lightlike submanifolds differs significantly from their non-degenerate counterparts. Such differences results from the fact that the  tangent bundle and the normal bundle of a lightlike submanifold have a non-trivial intersection. This makes the study of lightlike geometry extremely difficult compared to the non-degenerate case. In the books \cite{Duggal5, Duggal6} and \cite{Kupeli}, the authors introduce the geometry of lightlike submanifolds in semi-Riemannian manifolds, with relatively different approaches. Following the fundamental tools developed in the above books, many scholars have investigated the geometry of lightlike submanifolds. For instance, see, among others, the following articles; \cite{Aksu, Calin1, Duggal3, Duggal4, Bilal, Jin1, Jin2, Kang, Massamba1, Massamba2, ssekajja1, Sus}.

The geometry of lightlike hypersurfaces of indefinite Sasakian manifolds have been extensively investigated. The most investigated, among all, are the ones which are tangent to the structure vector field of the indefinite Sasakian manifold. For example, \cite{Jin2, ssekajja1} have shown that such hypersurfaces are never totally umbilical, totally $\eta$-umbilical, totally screen umbilical or screen conformal. Moreover, their screen distributions are never parallel, and their induced connections are never metric connections. Although the above hypersurfaces have also been studied under the condition that their second and local second fundamental forms are parallel, we have shown, in this paper, that such hypersurfaces can not be studied under such conditions. In fact, these hypersurfaces can not have parallel second fundamental forms nor parallel induced structure tensors (see Theorems \ref{theorem1}, \ref{theorem2}, \ref{theorem3}, \ref{theorem5}, \ref{theorem6} and \ref{theorem7}). The rest of the paper is arranged as follows; In Section \ref{pre}, we give some basic preliminaries on lightlike hypersurfaces required in the rest of the paper. In Section \ref{lightlike hypersurfaces}, we give some constructions on lightlike hypersurfaces of an indefinite Sasakian manifold. In Sections \ref{parasff} and \ref{recusff}, we discuss parallelism and recurrence of second fundamental forms. Finally, in Section \ref{parast}, we discuss the parallelism and recurrence of the induced structural tensors.

\section{Preliminaries} \label{pre}
Let $(\bar{M},\bar{g})$ be a $(2n+1)$-dimensional semi-Riemannian manifold with index $q$, where  $0< q < 2n+1$, and let $(M,g)$ be a hypersurface of $\bar{M}$. Let $g$ be the induced tensor field by $\bar{g}$ on $M$. Then, $M$ is called a {\it lightlike hypersurface} of $\bar{M}$ if $g$ is of constant rank $2n-1$ and the normal bundle $TM^{\perp}$ is a distribution of rank 1 on $M$ \cite{Duggal5, Duggal6}.  Here,  the fibres of the  vector bundle $TM^{\perp}$ are defined as $T_{x}M^{\perp}=\{Y_{x}\in T_{x}\bar{M}:\bar{g}_{x}(X_{x},Y_{x})=0\}$, for any $X_{x}$ tangent to $M$ and $x \in M$.  Let $M$  be a lightlike hypersurface of $(\bar{M},\bar{g})$. Consider the complementary distribution $S(TM)$ to $TM^{\perp}$ in $TM$, which is called a {\it screen distribution}. It is well-known that $S(TM)$ is non-degenerate (see \cite{Duggal5, Duggal6}). Thus,  we have  
\begin{align}\label{n97}
TM=S(TM)\perp TM^{\perp}.
\end{align}
As $S(TM)$ is non-degenerate with respect to $\bar{g}$, we have $T\bar{M}=S(TM)\perp S(TM)^{\perp}$, where $S(TM)^{\perp}$ is the complementary vector bundle to $S(TM)$ in $T\bar{M}|_{M}$. 

\begin{theorem}[Duggal-Bejancu \cite{Duggal5}]
Let $(M,g)$ be a lightlike hypersurface of $(\bar{M},\bar{g})$. Then, there exists a unique vector bundle $tr(TM)$, called the  lightlike transversal bundle of $M$ with respect to $S(TM)$,  of rank 1 over $M$ such that for any non-zero section $\xi$ of $TM^{\perp}$ on a coordinate neighbourhood $\mathcal{U}\subset M$, there exists a unique section $N$ of $tr(TM)$ on $\mathcal{U}$ satisfying 
\begin{align}\label{p5}
	\bar{g}(\xi,N)=1,\quad \bar{g}(N,N)=\bar{g}(N,Z)=0,
\end{align}
for any $Z$ tangent to $S(TM)$. 
\end{theorem}
Consequently, we have the following decomposition of $T\bar{M}$.  
\begin{align}
	T\bar{M}|_{M}&=S(TM)\perp \{TM^{\perp}\oplus tr(TM)\}=TM\oplus tr(TM).\label{n100}
\end{align}

Let $\nabla$, $\nabla^{t}$, $\nabla^{*}$ and  $\nabla^{*t}$ denote the induced connections on $M$,  $tr(TM)$, $S(TM)$ and $TM^{\perp}$, respectively, and $P$ be the projection of $TM$ onto $S(TM)$. Then the local Gauss-Weingarten equations of $M$ and $S(TM)$ are the following \cite{Duggal5, Duggal6}.
\begin{align}
 \bar{\nabla}_{X}Y&=\nabla_{X}Y+h(X,Y)=\nabla_{X}Y+B(X,Y)N,\label{int1}\\
 \bar{\nabla}_{X}N&=-A_{N}X+\nabla_{X}^{t}N=-A_{N}X+\tau(X)N,\label{flow6}\\
  \nabla_{X}PY&=\nabla^{*}_{X}PY+h^{*}(X,PY)= \nabla^{*}_{X}PY + C(X,PY)\xi,\label{int3}\\
  \nabla_{X}\xi &=-A^{*}_{\xi}X+\nabla_{X}^{*t}\xi=-A^{*}_{\xi}X -\tau(X) \xi,\label{flow7}
 \end{align}
for all $X$ and $Y$ tangent to $M$, $\xi$ tangent to $TM^{\perp}$ and $N$ tangent to $tr(T M)$, where $\bar{\nabla}$ is the Levi-Civita connection on $\bar{M}$. In the above setting, $h$ is called the second fundamental form and  $B$ the local second fundamental form of $M$. Furthermore, $h^{*}$ the second fundamental form and $C$ the local second fundamental form on $S(TM)$.  $A_{N}$ and $A^{*}_{\xi}$ are the shape operators of $TM$ and $S(TM)$ respectively, while $\tau$ is a 1-form on $TM$. 
 
 The above shape operators are related to their local fundamental forms by 
 \begin{align}
 	g(A^{*}_{\xi}X,Y) &=B(X,Y),\label{p9}\\
 	 g(A_{N}X,PY) &= C(X,PY), \nonumber
 \end{align}
 for any $X$ and $Y$ tangent to $M$. It follows from  (\ref{p9}) that 
 \begin{align}\label{p10}
 	B(X,\xi)=0,
 \end{align}
 for any $X$ tangent to $M$. Moreover, we have
 \begin{align}\label{p11}
 	\bar{g}(A^{*}_{\xi}X,N)=0\quad \mbox{and}\quad \bar{g}(A_{N} X,N) = 0, 
 \end{align}
 for all $X$ tangent to $M$. From  relations (\ref{p11}), we notice that $A_{\xi}^{*}$ and $A_{N}$ are both screen-valued operators. Since $S(TM)$ is non-degenerate (\ref{p10}) implies that 
 \begin{align}\label{n95}
 A_{\xi}^{*}\xi=0.
 \end{align}
 Let $\theta$ be a 1-form, on $M$, defined on $M$ by 
 \begin{align}\label{n96}
 \theta(X)=\bar{g}(X, N),
 \end{align}
 for all $X$ tangent to $M$. Then, in view of decomposition (\ref{n97}), any $X$ tangent to $M$ can be written as 
 \begin{align}\label{n98}
 X=PX+\theta(X)\xi.
 \end{align}
 
It is easy to show that 
\begin{align}\label{p40}
	(\nabla_{X}g)(Y,Z)=B(X,Y)\theta(Z)+B(X,Z)\theta(Y), 
\end{align}
for all $X$, $Y$ and $Z$ tangent to $M$. Consequently,  $\nabla$ is generally not a metric connection with respect to $g$. However, it is known that the induced connection $\nabla^{*}$ on $S(TM)$ is a metric connection.

On every lightlike hypersurface $(M,g)$ of a semi-Riemannian manifold $(\bar{M}, \bar{g})$, the covariant derivatives of the second fundamental form $h=B\otimes N$ and the local second fundamental form $B$ are defined as 
\begin{align}
(\nabla_{X}h)(Y,Z)&=\nabla^{t}_{X}h(Y,Z)-h(\nabla_{X}Y,Z)-h(Y, \nabla_{X}Z),\label{n102}\\
(\nabla_{X}B)(Y,Z)&=X\cdot B(Y,Z)-B(\nabla_{X}Y,Z)-B(Y, \nabla_{X}Z),\label{n103}
\end{align}
for any $X$, $Y$ and $Z$ tangent to $M$. From the fact  $\nabla^{t}_{X}N=\tau(X)N$, we have 
\begin{align}\label{x10}
\nabla^{t}_{X}h(Y,Z)&=X\cdot B(Y,Z)N+B(Y,Z)\nabla^{t}_{X}N\nonumber\\
&=\{X\cdot B(Y,Z)+B(Y,Z)\tau(X)\}N,
\end{align}
for any $X$, $Y$ and $Z$ tangent to $M$. In view of (\ref{n102}), (\ref{n103}) and (\ref{x10}), $h$ and $B$ satisfy the relation 
\begin{align}\label{x211}
(\nabla_{X}h)(Y,Z)=\{(\nabla_{X}B)(Y,Z)+\tau(X)B(Y,Z)\}N,
\end{align}
for any $X$, $Y$ and $Z$ tangent to $M$.

\section{Lightlike hypersurfaces of indefinite Sasakian manifolds}\label{lightlike hypersurfaces}
An odd-dimensional semi-Riemannian manifold $(\bar{M},\bar{g})$ is called contact metric manifold \cite{Duggal1} if there are a $(1, 1)$ tensor field $\bar{\phi}$, a vector field $\zeta$ , called structure vector field, and a 1-form $\eta$ such that
\begin{align}
	\bar{g}(\bar{\phi}X,\bar{\phi}Y)&=\bar{g}(X,Y)-\eta(X)\eta(Y),\quad \bar{g}(\zeta,\zeta)=1,\label{p2}\\
	\bar{\phi}^{2}X&=-X+\eta(X) \zeta, \quad \bar{g} (X,\zeta)=\eta(X),\label{p1}\\
	&d\eta(X,Y)=\bar{g}(X,\bar{\phi}Y),\nonumber
\end{align}
for any $X$ and $Y$ tangent to $\bar{M}$. It follows that $\bar{\phi}\zeta=0$, $\eta\circ \bar{\phi}=0$ and  $\eta(\zeta)=1$. Then $(\bar{\phi},\zeta,\eta,\bar{g})$ is called contact metric structure of $\bar{M}$. Furthermore, $\bar{M}$ has a normal contact structure if $N_{\bar{\phi}}+d\eta \otimes \zeta=0$, where $N_{\bar{\phi}}$ is the Nijenhuis tensor field \cite{Yano}. A normal contact metric $\bar{M}$ is called an indefinite Sasakian manifold \cite{Takahashi, Tanno}, for which we have
\begin{align}
	(\bar{\nabla}_{X}\bar{\phi})(Y) &=\bar{g}(X,Y)\zeta-\eta(Y)X,\label{k1}\\
	\bar{\nabla}_{X}\zeta &=-\bar{\phi}X,\label{p4}
\end{align}
for any $X$ and $Y$ tangent to $\bar{M}$.

Let $(M,g)$ be a lightlike hypersurface, tangent to the structure vector field $\zeta$, of an indefinite Sasakian manifold $(\bar{M}, \bar{\phi},\zeta, \eta, \bar{g})$.  In such a case, C. Calin \cite{Calin2} has shown that $\zeta$ belongs to the screen distribution $S(TM)$. Through out this paper, we assume that $\zeta$ belongs to $S(TM)$. Let $\xi$ and $N$ the metric normal and the transversal section, respectively. Since $(\bar{\phi},\zeta, \eta)$ is an almost contact structure and $\bar{\phi}\xi$ is a lightlike vector field, it follows that $\bar{\phi}N$ is lightlike too. Moreover, $\bar{g}(\bar{\phi}\xi,\xi)=0$ and, thus, $\bar{\phi}\xi$ is tangent to $TM$. Let us consider $S(TM)$ containing $\bar{\phi}TM^{\perp}$ as a vector subbundle. Consequently, $N$ is orthogonal to $\bar{\phi}\xi$ and we have $\bar{g}(\bar{\phi}N,\xi)=-\bar{g}(N,\bar{\phi}\xi)=0$ and $\bar{g}(\bar{\phi}N,N)=0$. This means that $\bar{\phi}N$ is tangent to $TM$ and in particular, it belongs to $S(TM)$. Thus, $\bar{\phi}tr(TM)$ is also a vector subbundle of $S(TM)$. In view of (\ref{p2}), we have
$\bar{g}(\bar{\phi}\xi,\bar{\phi}N)=1$. It is then easy to see that $\bar{\phi}TM^{\perp}\oplus \bar{\phi}tr(TM)$ is a non-degenerate vector subbbundle of $S(TM)$, with 2-dimensional fibers. Since $\zeta$ is tangent to $M$, and that $\bar{g}(\bar{\phi}\xi,\zeta)=\bar{g}(\bar{\phi}N,\zeta)=0$, then there exists a non-degenerate distribution $D_{0}$ on $TM$ such that 
\begin{align}\label{p13}
	S(TM)=\{\bar{\phi}TM^{\perp}\oplus \bar{\phi}tr(TM)\}\perp D_{0}\perp \langle\zeta\rangle,
\end{align}
where $\langle\zeta\rangle$ denotes the line bundle spanned by the structure vector field $\zeta$. Furthermore, one can easy to check that $D_{0}$ is an almost complex distribution with respect to $\bar{\phi}$, that is; $\bar{\phi}D_{0}=D_{0}$. 

Then, in view of (\ref{n97}), (\ref{n100}) and (\ref{p13}), the decompositions of $TM$ and $T\bar{M}$ becomes
\begin{align}\label{n101}
TM&=\{\bar{\phi}TM^{\perp}\oplus \bar{\phi}tr(TM)\}\perp D_{0}\perp \langle\zeta\rangle\perp TM^{\perp}.\\
T\bar{M}_{|M}&=\{\bar{\phi}TM^{\perp}\oplus \bar{\phi}tr(TM)\}\perp D_{0}\perp \langle\zeta\rangle\perp \{TM^{\perp}\oplus tr(TM)\}.\nonumber
\end{align}
 If we set $D=TM^{\perp}\perp \bar{\phi}TM^{\perp}\perp D_{0}$ and $D'=\bar{\phi}tr(TM)$, then (\ref{n101}) becomes
\begin{align}\label{p14}
	TM=D\oplus D'\perp \langle\zeta\rangle.
\end{align} 
Here, $D$ is an almost complex distribution and $D'$ is carried by $\bar{\phi}$ into the transversal bundle. 

Consider the lightlike vector fields $U$ and $V$ given by 
\begin{align}\label{f11}
	U=-\bar{\phi}N\quad \mbox{and}\quad V=-\bar{\phi}\xi, 
\end{align}
together with their corresponding $1$-forms $u$ and $v$ given by
\begin{align}\label{p16}
u(X)=g(X, V)\quad \mbox{and}\quad v(X)=g(X, U),
\end{align}
for any $X$ tangent to $M$.

Then, from (\ref{p14}), any $X$ tangent to $M$ can be written as 
\begin{align}\label{n90}
X=RX+QX+\eta(X)\zeta,
\end{align}
where $R$ and $Q$ are the projection morphisms of $TM$ onto $D$ and $D'$, respectively.

Applying $\bar{\phi}$ to (\ref{n90}) and using (\ref{p1}), we get 
\begin{align}\label{p17}
	\bar{\phi}X=\phi X+u(X)N,
\end{align}
where $\phi$ is a (1,1) tensor field defined on $M$ by $\phi X=\bar{\phi}RX$. Furthermore, we have 
\begin{align}\label{p18}
	\phi^{2}X&=-X+\eta(X)\zeta+u(X)U,\quad u(U)=1,\quad\phi U=0,\\
\eta(\phi X)&=u(\phi X)=0,\quad v(\phi X)=-\theta(X),\quad \theta(\phi X)=v(X),\label{p100}  	
\end{align}
for any $X$ tangent to $M$. It is easy to show that 
\begin{align}
	g(\phi X,\phi Y)&=g(X,Y)-\eta(X)\eta(Y)-u(Y)v(X)-u(X)v(Y),\nonumber\\
	g(\phi X,Y)&=-g(X,\phi Y)-u(X)\theta(Y)-u(Y)\theta(X),\label{p30}
\end{align}
for any $X$ and $Y$ tangent to $M$.

\begin{lemma}\label{lemma2}
	On a lightlike hypersurface $(M,g)$, tangent to the structure vector field $\zeta$, of an indefinite Sasakian manifold $(\bar{M}, \bar{\phi},\zeta, \eta, \bar{g})$, the following  holds.
\begin{align}
	\nabla_{X}\zeta &=-\phi X,\label{p19}\\
	B(X,\zeta)&=-u(X),\label{p20}\\
	C(X,\zeta)&=-v(X),\label{p21}\\
	B(X,U)&=C(X,V),\label{p22}\\
	(\nabla_{X}u)(Y)&=-B(X, \phi Y)-\tau(X)u(Y), \label{n110}\\
	(\nabla_{X}\phi)(Y)&=g(X,Y)\zeta-\eta(Y)X-B(X,Y)U+u(Y)A_{N}X,\label{p23}\\
	\nabla_{X}U&=\phi A_{N}X-\theta(X)\zeta+\tau(X)U,\label{p38}\\
		\nabla_{X}V&=\phi A^{*}_{\xi}X-\tau(X)V,\label{p39}
\end{align}
for all $X$ and $Y$ tangent to $M$.
\end{lemma}
\begin{proof}
	A proof uses straightforward calculations, while considering (\ref{k1}), (\ref{p4}) and (\ref{int1})--(\ref{flow7}).
\end{proof}
\begin{lemma}\label{lemmaq}
Let $(M,g)$ be a lightlike hypersurface, tangent to the structure vector field $\zeta$, of an indefinite Sasakian manifold $(\bar{M}, \bar{g}, \bar{\phi}, \eta, \zeta)$. Then  $A^{*}_{\xi}$ satisfy the relation
\begin{align}\label{n91}
A^{*}_{\xi}\zeta=-V.
\end{align}
\end{lemma}
\begin{proof}
A proof follows directly from relations (\ref{p9}) and (\ref{p20}), and the fact that $S(TM)$ is non-degenerate.
\end{proof}

\section{Parallelism of second fundamental forms $h$ and $B$}\label{parasff}

In this section, we show that there exist no lightlike hypersurface $(M,g)$, tangent to the structure vector field $\zeta$, of an indefinite Sasakian manifold $(\bar{M}, \bar{\phi},\zeta, \eta, \bar{g})$, can either have a parallel second fundamental form $h$ or a parallel local second fundamental form $B$. These findings can be seen in Theorems \ref{theorem1} and \ref{theorem2}. In order to establish these results, we need the following definition and some lemmas.

\begin{definition}
\rm{Let $(M,g)$ be a lightlike hypersurface of a semi-Riemannian manifold $(\bar{M}, \bar{g})$. Let $h$ be the second fundamental form of $M$, and $B$ its local second fundamental form. Then, we say that 
\begin{enumerate}
\item  $h$ is parallel if 
\begin{align}\label{n104}
(\nabla_{X}h)(Y,Z)=0;
\end{align}
\item  $B$ is parallel if 
\begin{align}\label{n105}
(\nabla_{X}B)(Y,Z)=0,
\end{align}
\end{enumerate}
for any $X$, $Y$ and $Z$ tangent to $M$}.
\end{definition}

\begin{remark}
\rm{One can easily see, from relation (\ref{x211}), that the parallelism of $h$ does not, in general, imply the parallelism of $B$ and vise-vasa. This motivates us to investigate their parallelism separately}.
\end{remark}

\begin{lemma}\label{lem1}
Let $(M,g)$ be a lightlike hypersurface of a semi-Riemannian manifold $(\bar{M},\bar{g})$. If either the second fundamental form $h$ of $M$ or its local second fundamental form $B$ is parallel, then $A^{*}_{\xi}$ satisfy 
\begin{align}\label{n2}
A^{*}_{\xi}A^{*}_{\xi}X=0,
\end{align}
for any $X$ tangent to $M$. 
\end{lemma}

\begin{proof}
Suppose that $h$ is parallel. Then, by (\ref{x211}) and (\ref{n104}), we have 
\begin{align}\label{n1}
(\nabla_{X}B)(Y,Z)+\tau(X)B(Y,Z)=0,
\end{align}
for any $X, Y$ and  $Z$ tangent to $M$. Then, replacing $Z$ with $\xi$ in (\ref{n1}), and using (\ref{p10}), we get 
\begin{align}\label{n3}
(\nabla_{X}B)(Y,\xi)=0, 
\end{align}
for any $X$ and $Y$ tangent to $M$. It follows from (\ref{n103}) and (\ref{n3}) that 
\begin{align}\label{n106}
X\cdot B(Y, \xi)-B(\nabla_{X}Y,\xi)-B(Y, \nabla_{X}\xi)=0.
\end{align}
Now, applying (\ref{p10}) to (\ref{n106}), we have 
\begin{align}\label{n107}
B(Y, \nabla_{X}\xi)=0. 
\end{align}
In view of (\ref{flow7}), (\ref{p10}) and (\ref{n107}), we have
$B(Y,A^{*}_{\xi}X)=0$. Thus, relation (\ref{n2}) follows easily from the last relation by the fact that $A^{*}_{\xi}$ is screen-valued operator and the non-degeneracy of the screen distribution $S(TM)$. In case $B$ is parallel, the proof also follows easily as above, which completes the proof.
\end{proof}

\begin{lemma}\label{lem2}
Let $(M,g)$ be a lightlike hypersurface, tangent to the structure vector field $\zeta$, of an indefinite Sasakian  manifold $(\bar{M}, \bar{\phi},\zeta, \eta, \bar{g})$, with a parallel second fundamental form $h$,  then   
\begin{align}
&B(X,\phi Y)+B(Y, \phi X)=0,\label{n5}\\
&\;A^{*}_{\xi}V=0, \quad A^{*}_{\xi}\phi A^{*}_{\xi} X=0,\label{n6}
\end{align}
for any $X$ and $Y$ tangent to $M$. 
\end{lemma}

\begin{proof}
As $h$ is parallel, we replace $Z$ with $\zeta$ in  (\ref{x211}) to get 
\begin{align*}
(\nabla_{X}B)(Y,\zeta)+\tau(X)B(Y,\zeta)=0,
\end{align*}
for any $X$ and  $Y$  tangent to $M$, which after applying (\ref{p20}) to the second term, reduces to 
\begin{align}\label{n4}
(\nabla_{X}B)(Y,\zeta)-u(Y)\tau(X)=0.
\end{align}
But, using  (\ref{n103}), (\ref{p19}), (\ref{p20}) and (\ref{n110}), we have 
\begin{align}\label{n7}
(\nabla_{X}B)(Y,\zeta)&=X\cdot B(Y,\zeta)-B(\nabla_{X}Y,\zeta)-B(Y, \nabla_{X}\zeta)\nonumber\\
&=-X\cdot u(Y)+u(\nabla_{X}Y)+B(Y,\phi X)\nonumber\\
&=-(\nabla_{X}u)(Y)+B(Y,\phi X)\nonumber\\
&=B(X,\phi Y)+u(Y)\tau(X)+B(Y,\phi X).
\end{align}
Then, relation (\ref{n5}) follows from (\ref{n4}) and (\ref{n7}). 

Next, replacing $Y$ with $\xi$ in (\ref{n5}) and noting that $\phi \xi =-V$, we get $B(X,V)=0$, which leads to $A^{*}_{\xi}V=0$ by the non-degeneracy of the screen distribution $S(TM)$, and hence the first relation in (\ref{n6}) is proved.  Furthermore, replacing $Y$ with $A^{*}_{\xi}Y$ in (\ref{n5}) and using (\ref{n2}) of Lemma \ref{lem1}, we have $B(X, \phi A^{*}_{\xi}Y)=0$. This leads to $A^{*}_{\xi}\phi A^{*}_{\xi}Y=0$ by the non-degeneracy of the screen distribution $S(TM)$. This proves the second relation in (\ref{n6}), which completes the proof.  
\end{proof}

 In view of Lemmas \ref{lem1} and  \ref{lem2}, we have the following result. 

\begin{theorem}\label{theorem1}
There exist no lightlike hypersurface $(M,g)$, tangent to the structure vector field $\zeta$, of an indefinite Sasakian manifold $(\bar{M}, \bar{\phi},\zeta, \eta, \bar{g})$, with a parallel second fundamental form $h$.
\end{theorem}

\begin{proof}
In view of relation (\ref{p9}) and (\ref{n5}) of Lemma \ref{lem2}, we have
\begin{align}\label{n9}
	g(A^{*}_{\xi}X,\phi Y)+g(A^{*}_{\xi}\phi X,Y)=0,
\end{align}
for any $X$ and $Y$ tangent to $M$. Applying relation (\ref{p30}) and the first relation in (\ref{n6}) to the first term in (\ref{n9}), we derive 
\begin{align}\label{n10b}
	g(A^{*}_{\xi}X,\phi Y)&=-g(\phi A^{*}_{\xi}X, Y)-B(X, V)\theta(Y)=-g(\phi A^{*}_{\xi}X, Y).
\end{align}
Replacing (\ref{n10b}) in (\ref{n9}), and using the fact that $S(TM)$ is non-degenerate, we get 
\begin{align}\label{n111}
	-P\phi A^{*}_{\xi}X+A^{*}_{\xi}\phi X=0,
\end{align}
for any $X$ tangent to $M$. On the other hand, using the last relation in (\ref{p100}), we have 
\begin{align}\label{x1}
\theta(\phi A^{*}_{\xi}X)=v(A^{*}_{\xi}X)=B(X,U),
\end{align}
for any $X$ tangent to $M$. Therefore, by the help of (\ref{n96}), (\ref{n98}), (\ref{f11}), (\ref{p17}) and (\ref{x1}), we have  
\begin{align}
P\phi A^{*}_{\xi}X&=\phi A^{*}_{\xi}X-\theta(\phi A^{*}_{\xi}X)\xi=\phi A^{*}_{\xi}X-B(X,U)\xi.\label{n112}
\end{align}
 Replacing (\ref{n112}) in (\ref{n111}), we get 
\begin{align}\label{n11}
	-\phi A^{*}_{\xi}X+B(X,U)\xi+A^{*}_{\xi}\phi X=0.
\end{align}
Applying $\phi$ to (\ref{n11}) and using (\ref{p18}), and remembering that $B(X,V)=0$ and $\phi \xi=-V$, we get 
\begin{align}\label{n12}
	A^{*}_{\xi}X-B(X,\zeta)\zeta-B(X,U)V+\phi A^{*}_{\xi}\phi X=0.
\end{align}
Next, applying $A^{*}_{\xi}$ to (\ref{n12}) and using (\ref{p20}), (\ref{n91}), (\ref{n2}) and the two relations in  (\ref{n6}), we get 
\begin{align}\label{n13}
	-B(X,\zeta)A^{*}_{\xi}\zeta=B(X,\zeta)V=-u(X)V=0,
\end{align}
for any $X$ tangent to $M$. Taking $U$ for $X$ in (\ref{n13}), we get $V=0$ which is impossible, which  completes the proof.
\end{proof}

 Next, we turn our attention to the parallelism of the local second fundamental form $B$. In that line, we need the following lemma. 

\begin{lemma}\label{lemma2}
	Let $(M,g)$ be a lightlike hypersurface, tangent to the structure vector field $\zeta$, of an indefinite Sasakian manifold $(\bar{M}, \bar{g}, \bar{\phi}, \eta, \zeta, \bar{g})$, with a parallel local second fundamental form $B$, then 
	\begin{align}
		B(X,\phi &Y)+B(Y, \phi X)+\tau(X)u(Y)=0,\label{n21}\\
		&A^{*}_{\xi}V=0, \quad A^{*}_{\xi}\phi A^{*}_{\xi}X=0,\label{n22}
	\end{align}
	for any $X$ and $Y$ tangent to $M$.
\end{lemma}
\begin{proof}
	Suppose that $B$ is parallel.  Then by (\ref{n105}), we have
	\begin{align}\label{n17}
		(\nabla_{X}B)(Y,Z)=0,
	\end{align}
	for any $X$, $Y$ and $Z$ tangent to $M$. Replacing $Z$ with $\zeta$ in (\ref{n17}) and using (\ref{n103}), we get 
\begin{align}\label{n19}
		X\cdot B(Y,\zeta)-B(\nabla_{X}Y,\zeta)-B(Y, \nabla_{X}\zeta)=0.
	\end{align}
	Applying (\ref{p19}) and (\ref{p20}) to (\ref{n19}), we get 
	\begin{align}\nonumber
		-X\cdot u(Y)+u(\nabla_{X}Y)+B(Y,\phi X)=0,
	\end{align}
	which simplifies to 
	\begin{align}\label{n20}
		-(\nabla_{X}u)(Y)+B(Y,\phi X)=0.
	\end{align}
	Substituting (\ref{n110}) in (\ref{n20}), we get 
	\begin{align}\nonumber
		B(X, \phi Y)+\tau(X)u(Y)+B(Y,\phi X)=0,
	\end{align}
which proves (\ref{n21}). Replacing $Y$ with $\xi$ in (\ref{n21}) and noting that $\phi \xi=-V$, we get $B(X,V)=0$, which gives $A^{*}_{\xi}V=0$. On the other hand, replacing $Y$ with $A^{*}_{\xi}Y$ in (\ref{n21}) and using (\ref{n2}) of Lemma \ref{lem1} and the first relation in (\ref{n22}), we get $B(X,\phi A^{*}_{\xi}Y)=0$, which completes the proof.
\end{proof}

 Using Lemma \ref{lemma2}, we prove the following result.

\begin{theorem}\label{theorem2}
There exist no lightlike hypersurface $(M,g)$, tangent to the structure vector field $\zeta$, of an indefinite Sasakian manifold $(\bar{M}, \bar{\phi},\zeta, \eta, \bar{g})$, with a parallel local second fundamental form $B$.
\end{theorem}

\begin{proof}
	From (\ref{p9}), (\ref{p16}) and (\ref{n21}), we have 
	\begin{align}\label{n30}
		g(A^{*}_{\xi}X,\phi Y)+g(A^{*}_{\xi}Y, \phi X)+g(\tau(X)V,Y)=0,
	\end{align}
for any $X$ and $Y$ tangent to $M$. Since $B(X,V)=0$ (see Lemma \ref{lemma2}), we simplify the first term in (\ref{n30}), using relation (\ref{p30}), as follows 
	\begin{align}\label{n31}
		g(A^{*}_{\xi}X,\phi Y)&=-g(\phi A^{*}_{\xi}X, Y)-B(X,V)\theta(Y)=-g(\phi A^{*}_{\xi}X, Y).		
\end{align}
Replacing (\ref{n31}) in (\ref{n30}) and using the fact the screen distribution $S(TM)$ is non-degenerate, we get 
\begin{align}\label{n32}
	-P\phi A^{*}_{\xi}X+A^{*}_{\xi}\phi X+\tau(X)V=0
\end{align}
It is easy to see, from the last relation in (\ref{p100}), that 
\begin{align}\label{x4}
\theta(\phi A^{*}_{\xi}X)=v(A^{*}_{\xi}X)=B(X,U),
\end{align}
for any $X$ tangent to $M$. Thus, applying (\ref{n98}) and (\ref{x4}), we have
\begin{align}\label{x6}
P\phi A^{*}_{\xi}X=\phi A^{*}_{\xi}X-\theta(A^{*}_{\xi}X)\xi=\phi A^{*}_{\xi}X-B(X,U)\xi,
\end{align}
for any $X$ tangent to $M$. Then replacing (\ref{x6}) in  (\ref{n32}), we get 
\begin{align}\label{n33}
	-\phi A^{*}_{\xi}X+B(X,U)\xi+A^{*}_{\xi}\phi X+\tau(X)V=0.
\end{align}
Applying $\phi$ to (\ref{n33}) and using (\ref{p18}), we derive 
\begin{align}\label{n34}
	A^{*}_{\xi}X-\eta(A^{*}_{\xi}X)\zeta-u(A^{*}_{\xi}X)U+B(X,U)\phi \xi+\phi A^{*}_{\xi}\phi X+\tau(X)\phi V=0.
\end{align}
But $\eta(A^{*}_{\xi}X)=B(X,\zeta)$, $u(A^{*}_{\xi}X)=B(X,V)=0$, $\phi \xi =-V$	and $\phi V=\xi$. Therefore, (\ref{n34}) reduces to 
\begin{align}\label{n35}
	A^{*}_{\xi}X-B(X, \zeta)\zeta-B(X,U)V+\phi A^{*}_{\xi}\phi X+\tau(X)\xi=0.
\end{align}
Applying $A^{*}_{\xi}$ to (\ref{n35}) and using   (\ref{n98}), (\ref{p20}), (\ref{n2}) and  (\ref{n22}) of Lemma \ref{lemma2}, together with (\ref{n91}) of Lemma \ref{lemmaq}, we get 
\begin{align}\label{n36}
	-B(X,\zeta)A^{*}_{\xi}\zeta=B(X,\zeta)V=-u(X)V=0,
\end{align}
for any $X$ tangent to $M$. Replacing $X$ with $U$ in (\ref{n36}), we get $V=0$ which is impossible, and hence the proof.
\end{proof}

\begin{remark}
	\rm{Lightlike hypersurfaces $(M,g)$, tangent to the structure vector field $\zeta$, of an indefinite Sasakian space form $(\bar{M}(c), \bar{\phi},\zeta, \eta, \bar{g})$,  with parallel second fundamental form $h$ and parallel local second fundamental form $B$ have been studied in \cite[pages 346--352]{Massamba1} and in \cite[pages 231--233]{Massamba2}.  We however,  stress, with the help of Theorems \ref{theorem1} and \ref{theorem2} above,  that such hypersurfaces do not exist}.
\end{remark}

\section{Recurrence of second fundamental forms $h$ and $B$}\label{recusff}

Away from the second fundamental form $h$ and the local second fundamental form $B$ being parallel, we investigate a more general case, in which the above tensors are recurrent. For this reason, we start with the following definition. 

\begin{definition}
\rm{	Let $(M,g)$ be a lightlike hypersurface of a semi-Riemannian manifold $(\bar{M}, \bar{g})$. Let $h$ and $B$ denote its second fundamental form and local second fundamental form, respectively. Then, we say that
	\begin{enumerate}
		\item  $h$ is recurrent if there exist a one-form $\alpha$ on $M$ such that 
		\begin{align}\label{n44}
			(\nabla_{X}h)(Y,Z)=\alpha (X)h(Y,Z);
		\end{align}
		\item  $B$ is recurrent if there exist a one-form $\beta$ on $M$ such that 
		\begin{align}\label{n45}
			(\nabla_{X}B)(Y,Z)=\beta (X)B(Y,Z),
		\end{align}
	\end{enumerate}
	for any $X$, $Y$ and $Z$ tangent to $M$. Moreover, if $\alpha=\beta=0$ then we see that $h$ and $B$ becomes parallel}.
\end{definition}

\begin{proposition}\label{proposition1}
Let $(M,g)$ be a lightlike hypersurface of a semi-Riemannian manifold $(\bar{M}, \bar{g})$. Then, the second fundamental form $h$ of $M$ is recurrent if and only if its  local second fundamental form $B$ is recurrent.
\end{proposition}

\begin{proof}
Suppose that $h$ is recurrent. Then,  by (\ref{x211}) and (\ref{n44}),  we have 
\begin{align*}
(\nabla_{X}B)(Y,Z)=\{\alpha(X)-\tau(X)\}B(Y,Z),
\end{align*}
for any $X$, $Y$ and $Z$ tangent to $M$. This shows that $B$ is also recurrent with $\beta=\alpha-\tau$. On the other hand, if $B$  is recurrent, then (\ref{x211}) and (\ref{n45}) gives 
\begin{align*}
(\nabla_{X}h)(Y,Z)= \{\beta(X)+\tau(X)\}B(Y,Z)N=\{\beta(X)+\tau(X)\}h(Y,Z),
\end{align*}
for any $X$, $Y$ and $Z$ tangent to $M$, This shows that $h$ is also recurrent with $\alpha=\beta+\tau$, which completes the proof.
\end{proof}
 In view of Proposition \ref{proposition1}, it is enough to investigate the recurrence of either $h$ or $B$. Here, we shall investigate the recurrence of the second fundametal form $h$. Let us start with the following lemma.

\begin{lemma}\label{lemma6}
Let $(M,g)$ be a lightlike hypersurface, tangent to structure vector field $\zeta$, of a semi-Riemannian manifold $(\bar{M},\bar{g})$. If the   second fundamental form $h$ of  $M$ is recurrent, then 
	\begin{align}\label{n43}
		A^{*}_{\xi}A^{*}_{\xi}X=0,
	\end{align}
	for any $X$ tangent to $M$. Furthermore, if $\bar{M}=(\bar{M}, \bar{g}, \bar{\phi}, \eta, \zeta, \bar{g})$ is an indefinite Sasakian manifold, such that the lightlike hypersurface $M$ is tangent to the structure vector field $\zeta$, then 
	\begin{align}
	&\alpha(X)=-B(U,\phi X),\label{n40}\\
	B(X,\phi Y)&+B(Y, \phi X)-B(U, \phi X)u(Y)=0,\label{n41}\\
		A^{*}_{\xi}&V=0,\quad A^{*}_{\xi}\phi A^{*}_{\xi}X=0,\label{n42}
	\end{align}
	for any $X$ and $Y$ tangent to $M$.
\end{lemma}
\begin{proof}
	Using (\ref{x211}) and (\ref{n44}), we get 
	\begin{align}\label{n46}
		(\nabla_{X}B)(Y,Z)+\tau(X)B(Y,Z)=\alpha(X)B(Y,Z).
	\end{align}
	Replacing $Z$ with $\xi$ in (\ref{n46}) and using (\ref{p10}), we get $(\nabla_{X}B)(Y,\xi)=0$. With the help of relations (\ref{flow7}), (\ref{p10}) and (\ref{n103}), the last relation becomes $B(Y,A^{*}_{\xi}X)=0$, from which (\ref{n43}) follows easily by the non-degeneracy of $S(TM)$. 
	
Next, if $\bar{M}=(\bar{M}, \bar{g}, \bar{\phi}, \eta, \zeta, \bar{g})$ is an indefinite Sasakian manifold such that the lightlike hypersurface $M$ is tangent to the structure vector field $\zeta$, then we can replace $Z$ with $\zeta$ in (\ref{n46}), and using (\ref{p20}), to get
\begin{align}\label{p102}
(\nabla_{X}B)(Y,\zeta)-\tau(X)u(Y)=-\alpha(X)u(Y),
\end{align}
for any $X$ and $Y$ tangent to $M$. On the other hand, using  (\ref{n103}), (\ref{p19}), (\ref{p20}) and (\ref{n110}), we have 
\begin{align}\label{p103}
(\nabla_{X}B)(Y,\zeta)&=X\cdot B(Y,\zeta)-B(\nabla_{X}Y,\zeta)-B(Y, \nabla_{X}\zeta)\nonumber\\
&=-X\cdot u(Y)+u(\nabla_{X}Y)+B(Y,\phi X)\nonumber\\
&=-(\nabla_{X}u)(Y)+B(Y,\phi X)\nonumber\\
&=B(X,\phi Y)+u(Y)\tau(X)+B(Y,\phi X).
\end{align}
Replacing (\ref{p103}) in (\ref{p102}), we get 
	\begin{align}\label{n48}
		B(X, \phi Y)+B(Y, \phi X)+\alpha (X)u(Y)=0.
	\end{align}
Taking $U$ instead of $Y$ in (\ref{n48}) and noting that $\phi U=0$, we have $\alpha(X)=-B(U,\phi X)$, which proves (\ref{n40}). Then, relation (\ref{n41}) follows from (\ref{n40}) and (\ref{n48}). Next, replacing $Y$ with $\xi$ in (\ref{n41}), we get $B(X,\phi \xi)=-B(X,V)=0$, i.e. $A^{*}_{\xi}V=0$, proving the first relation in (\ref{n42}). Next, replacing $Y$ with $A^{*}_{\xi}Y$ in (\ref{n41}) and using relations (\ref{p9}), (\ref{n43}) and the first relation in (\ref{n42}), we get $B(X,\phi A^{*}_{\xi}Y)=0$. This give us $A^{*}_{\xi}\phi A^{*}_{\xi}X=0$, by the non-degeneracy of $S(TM)$, which completes the proof.
\end{proof}

 With the help of Lemma \ref{lemma6}, we have the following result.

\begin{theorem}\label{theorem3}
	There exist no lightlike hypersurface $(M,g)$, tangent to the structure vector field $\zeta$,  of an indefinite Sasakian manifold $(\bar{M}, \bar{\phi},\zeta, \eta, \bar{g})$, with a recurrent second fundamental form $h$.
\end{theorem}

\begin{proof}
From (\ref{n42}), we have  $B(X,V)=0$, for any $X$ tangent to $M$. Therefore, we simplify  the first term in (\ref{n41}), with the help of relation (\ref{p30}), as follows 
	\begin{align}\label{n319}
		g(A^{*}_{\xi}X,\phi Y)=-g(\phi A^{*}_{\xi}X, Y)-B(X,V)\theta(Y)=-g(\phi A^{*}_{\xi}X, Y),		
\end{align}
for any $X$ and $Y$ tangent to $M$. Then, using relations (\ref{p9}), (\ref{n41}) and (\ref{n319}), and the non-degeneracy of $S(TM)$, we have 
	\begin{align}\label{n501}
		-P\phi A^{*}_{\xi}X+A^{*}_{\xi}\phi X-B(U,\phi X)V=0,
	\end{align}
	for any $X$ tangent to $M$. 
By a straightforward calculation, while considering ther last relation in (\ref{p100}), we have 
\begin{align}\label{x5}
\theta(\phi A^{*}_{\xi}X)=v(A^{*}_{\xi}X)=B(X,U),
\end{align}
for any $X$ tangent to $M$.	In view of (\ref{n98}) and (\ref{x5}), we have
\begin{align}\label{x8}
P\phi A^{*}_{\xi}X&=\phi A^{*}_{\xi}X-\theta(\phi A^{*}_{\xi}X)\xi=\phi A^{*}_{\xi}X-B(X,U)\xi,
\end{align}
for any $X$ tangent to $M$. Now, relation (\ref{n501}) can be re-written, using  (\ref{x8}), as 
	\begin{align}\label{n510}
		-\phi A^{*}_{\xi}X+B(X, U)\xi+A^{*}_{\xi}\phi X-B(U,\phi X)V=0.
	\end{align}
	Applying $\phi$ to (\ref{n510}) and using (\ref{p18}), while noting that $u(A^{*}_{\xi}X)=B(X,V)=0$, $\phi \xi =-V$ and $\phi V=\xi$, we get 
	\begin{align}\label{n52a}
		A^{*}_{\xi}X-B(X,\zeta)\zeta-B(X,U)V+\phi A^{*}_{\xi}\phi X-B(U,\phi X)\xi=0.
	\end{align}
	Then, applying $A^{*}_{\xi}$ to (\ref{n52a}), while using (\ref{n95}), (\ref{p20}),  (\ref{n91}), (\ref{n43}) and (\ref{n42}), we get 
	\begin{align}\label{n53a}
		-B(X,\zeta)A^{*}_{\xi}\zeta=B(X,\zeta)V=-u(X)V=0.
	\end{align}
Taking $U$ instead of $X$  in (\ref{n53a}), we get $V=0$, which is impossible. This completes the proof.
\end{proof}

\begin{corollary}\label{corollary4} 
	There exist no lightlike hypersurface $(M,g)$, tangent to the structure vector field $\zeta$, of an indefinite Sasakian manifold $(\bar{M}, \bar{\phi},\zeta, \eta, \bar{g})$, with a recurrent local second fundamental form $B$.
\end{corollary}

\section{Parallelism of the induced structures $\phi$, $U$ and $V$}\label{parast}

 Turning to the induced structures $\phi$, $U$ and $V$, we have the following results.

\begin{theorem}\label{theorem5}
There exist no lightlike hypersurface $(M,g)$,  tangent to the structure vector field $\zeta$,  of an indefinite Sasakian manifold $(\bar{M}, \bar{\phi},\zeta, \eta, \bar{g})$, with  parallel induced structures $\phi$, $U$ and $V$.
\end{theorem}

\begin{proof}
	Suppose that $\phi$ is parallel, i.e. $\nabla \phi=0$. Then, relation (\ref{p23}) gives 
	\begin{align}\label{n70}
		g(X,Y)\zeta-\eta(Y)X-B(X,Y)U+u(Y)A_{N}X=0,
	\end{align}
	for any $X$ and $Y$ tangent to $M$. Replacing $Y$ with $\zeta$ in (\ref{n70}) and using (\ref{p20}), we get 
	\begin{align}\label{n71}
		\eta(X)\zeta-X+u(X)U=0.
	\end{align}
Taking $V$ for $X$ in (\ref{n71}), we get the obvious contradiction $V=0$. 
	
Next, suppose $\nabla _{X}U=0$, for any $X$ tangent to $M$. Then, relation (\ref{p38}) gives 
	\begin{align}\label{n72}
		 \phi A_{N}X-\theta(X)\zeta+\tau(X)U=0,
	\end{align}
for any $X$ tangent to $M$. The inner product of (\ref{n72}) with $\zeta$ gives $\theta(X)=0$, which is impposisibe since $\theta(\xi)=\bar{g}(\xi,N)=1$. 

Finally if $\nabla_{X} V=0$, for any $X$ tangent to $M$, then by (\ref{int3}) one gets $\nabla^{*}_{X}V=0$  and $C(X,V)=0$, for any $X$ tangent to $M$. With the help of relation (\ref{p22}), the last relation means that $B(X,U)=C(X,V)=0$. If we replace $X$  with $\zeta$ in this relation and considering (\ref{p20}), we get $0=B(U,\zeta)=-u(U)=-1$,
which is impossible.
\end{proof}

\begin{definition}\cite[page 1922]{Jin1}
\rm{The structure tensor field $\phi$ of $(M,g)$ is said to be recurrent if there exists a 1-form $\omega$ on $M$ such that 
\begin{align}\label{n73}
	(\nabla _{X}\phi)(Y)=\omega(X)\phi Y,
\end{align}
for any $X$ and $Y$ tangent to $M$. $\phi$ is parallel whenever $\omega=0$}. 
\end{definition}

 As a generalization of part of Theorem \ref{theorem5}, we have the following result.
\begin{theorem}\label{theorem6}
	There exist no lightlike hypersurface $(M,g)$,  tangent to the structure vector field $\zeta$, of an indefinite Sasakian manifold $(\bar{M},\bar{\phi},\zeta, \eta, \bar{g})$, with a recurrent induced structure tensor $\phi$.
\end{theorem}
\begin{proof}
	Suppose that $\phi$ is recurrent. Then, relation (\ref{p23}) and (\ref{n73}) gives 
	\begin{align}\label{n74}
		g(X,Y)\zeta-\eta(Y)X-B(X,Y)U+u(Y)A_{N}X=\omega(X)\phi Y,
	\end{align}
	for any $X$ and $Y$ tangent to $M$. Replacing $Y$ with $\zeta$ in (\ref{n74}) and using (\ref{p20}) and the fact $\phi \zeta=0$, we get 
	\begin{align}\label{n75}
		\eta(X)\zeta-X+u(X)U=0.
	\end{align}
Replacing $X$ with $V$ in (\ref{n75}), we get the obvious impossibility $V=0$.
\end{proof}

 We also have the following.

\begin{theorem}\label{theorem7}
There exist no lightlike hypersurface $(M,g)$, tangent to the structure vector field $\zeta$ , of an indefinite Sasakian manifold $(\bar{M}, \bar{\phi},\zeta, \eta, \bar{g})$, such that $D'$ is a killing distribution.
\end{theorem}
\begin{proof}
By a direct calculation, we have 
\begin{align}\label{n80}
(\mathcal{L}_{U}g)(X,Y)&=U\cdot g(X,Y)-g([U,X],Y)-g(X, [U,Y])\nonumber\\
&=(\nabla_{U}g)(X,Y)+g(\nabla_{X}U,Y)+g(X,\nabla_{Y}U),
\end{align}
for any $X$ and $Y$ tangent to $M$. Then, applying (\ref{p40}) to (\ref{n80}), we get 
\begin{align}\label{n81}
(\mathcal{L}_{U}g)&(X,Y)\nonumber\\
&=B(U,Y)\theta(X)&+B(U,X)\theta(Y)+g(\nabla_{X}U,Y)+g(X,\nabla_{Y}U).
\end{align}
Suppose that $\mathcal{L}_{U}g=0$, that is $D'$ is a killing distribution, then (\ref{n81}) leads to 
\begin{align}\label{n82}
B(U,Y)\theta(X)&+B(U,X)\theta(Y)+g(\nabla_{X}U,Y)+g(X,\nabla_{Y}U)=0.
\end{align}
Replacing $X$ with $\xi$ and $Y$ with $\zeta$ in (\ref{n82}) and using (\ref{p10}), we get 
\begin{align}\label{n83}
B(U, \zeta)+g(\nabla_{\xi}U,\zeta)=0.
\end{align}
But, using (\ref{p19}) and (\ref{p20}), we have 
\begin{align}
B(U,\zeta)&=-u(U)=-1,\label{s1}\\
g(\nabla_{\xi}U,\zeta)&=-g(U, \nabla_{\xi}\zeta)=g(U,\phi \xi)=-1.\label{s2}
\end{align}
Replacing (\ref{s1}) and (\ref{s2}) in (\ref{n83}), we arrive at an obvious contradiction.
\end{proof}

\begin{remark}
\rm{ In the paper \cite[Theorem 4.11]{Massamba2}, the author studies lightlike hypersurfaces $(M,g)$, tangent to the structure vector field $\zeta$, of an indefinite Sasakian manifold $(\bar{M}, \bar{\phi},\zeta, \eta, \bar{g})$, with the assumption that $U$ and $V$ are parallel vector fields. But this is not possible by Theorem \ref{theorem5}. Furthermore, in the paper \cite[Theorem 4.12]{Massamba1}, the ditribution $D'$ is assumed to be killing but by Theorem \ref{theorem7} this is not possible. On the other hand, in the paper \cite[Theorem 1]{Sus}, the author studies the geometry  of recurrent lightlike hypersurfaces of an indefinite Sasakian manifold. However, we have seen, in Theorem \ref{theorem6}, that such hypersurfaces do not exist}.
\end{remark}

\end{document}